\documentclass[oneside,11pt,reqno]{amsart}
\usepackage{stmaryrd}
\usepackage{bbm}
\usepackage{mathrsfs}
\usepackage{latexsym,amsxtra,xypic}
\usepackage[dvips]{graphicx}
\usepackage{dsfont}
\usepackage[all]{xy}
\usepackage{amscd,graphics}
\usepackage{amsmath,amsfonts,amsthm,amssymb}
\usepackage{latexsym,amsmath}
\usepackage{graphicx,psfrag}

\newtheorem{thm}{Theorem}[section]
\newtheorem{lem}[thm]{Lemma}
\newtheorem{cor}[thm]{Corollary}
\newtheorem{pro}[thm]{Proposition}
\newtheorem{ex}[thm]{Example}
\newtheorem{rmk}[thm]{Remark}
\newtheorem{defi}[thm]{Definition}

\newtheorem{nota}[thm]{Notation}
\title
{SU(2)-Cyclic Surgeries on Knots}

\author{Jianfeng Lin}

\begin{document}

\maketitle 
\begin{abstract}
 A surgery on a knot in $S^{3}$ is called $SU(2)$-cyclic if it gives a manifold whose fundamental group has no non-cyclic $SU(2)$ representations. Using holonomy perturbations on the Chern-Simons functional, we prove that two $SU(2)$-cyclic surgery coefficients $\frac{p_{1}}{q_{1}}$ and $\frac{p_{2}}{q_{2}}$ should satisfy $|p_{1}q_{2}-p_{2}q_{1}|\leq |p_{1}|+|p_{2}|$. This is an analog of Culler-Gordon-Luecke-Shalen's cyclic surgery theorem.
\end{abstract}

\section{Introduction}

\begin{defi}
 A closed orientable $3$ manifold $M$, $M$ is called $SU(2)$-cyclic (or $SO(3)$-cyclic) if there exists no homomorphism $\phi: \pi_{1}(M)\rightarrow SU(2) \text{ (or }SO(3))$ with non-cyclic image.
\end{defi}

Suppose $K\subset S^{3}$ is a knot. For $r\in \mathds{Q}$, we denote the manifold obtained by doing $r$-surgery on $K$ by $K(r)$.

\begin{defi}
A surgery on $K$ with coefficient $r$ is called $SU(2)$-cyclic (or $SO(3)$-cyclic) if $K(r)$ is $SU(2)$-cyclic (or $SO(3)$-cyclic).
\end{defi}

We have the following exact sequence:
$$0\rightarrow \mathbb{Z}_{2}\rightarrow SU(2)\rightarrow SO(3)\rightarrow 0$$

It's easy to see that an $SO(3)$-cyclic surgery is always an $SU(2)$-cyclic surgery. Using some basic obstruction theory, we get:

\begin{lem}\label{homomorphism lift}
If $r=\frac{p}{q}$ is an $SU(2)$-cyclic surgery with $p$ odd, then $r$ is an $SO(3)$-cyclic surgery.
\end{lem}

In \cite{KM}, Kronheimer and Mrowka proved the following theorem:

\begin{thm}[Kronheimer, Mrowka 2003 \cite{KM}]\label{|r|>2}
Any $r$-surgery on a nontrivial knot with surgery coefficient $|r|\leq2$ is not $SU(2)$-cyclic.
\end{thm}

In particular, this theorem gave a proof for the Property-P Conjecture:

\begin{cor}
[Kronheimer, Mrowka 2003 \cite{KM}] A nontrivial surgery on a nontrivial knot does not give simply connected $3$-manifold.
\end{cor}

Obviously, lens spaces are all $SU(2)$-cyclic and $SO(3)$-cyclic. Thus all cyclic surgeries (the surgeries which give lens spaces) are $SO(3)$-cyclic. Therefore, we have:

\begin{ex}\label{torus knot}
The $pq-\frac{1}{r}$ $(r\in \mathds{Z})$ surgeries on the $(p,q)$-torus knot are cyclic and hence $SO(3)$-cyclic.
\end{ex}

Dunfield \cite{Dunfield} gives the following example:
\begin{ex}\label {pretzel knot}
The  $18, \frac{37}{2}, 19$ surgeries on the $(-2,3,7)$-pretzel knot are $SO(3)$-cyclic. The $18,19$ surgeries give lens spaces, while $K(\frac{37}{2})$ is a graph manifold obtained by gluing the left-handed trefoil knot complement and the right-handed trefoil complement.
\end{ex}

Another related theorem is Culler-Gordon-Luecke-Shalen's cyclic surgery theorem (we only state the case for knot surgery):

\begin{thm}
[Culler-Gordon-Luecke-Shalen \cite{CGLS}]\label{Cyclic} Suppose that $K$ is not a torus knot and $r,s$ are both cyclic surgeries, then $\triangle (r,s)\leq1$.
\end{thm}

Here for two rational numbers $r=\frac{p_{1}}{q_{1}}$ and $s=\frac{p_{2}}{q_{2}}$, the distance  $\triangle(r,s)$ is defined to be $|p_{1}q_{2}-p_{2}q_{1}|$.

Since $\frac{1}{0}$-surgery is always cyclic, this theorem implies that when $K$ is not a torus knot, $r$-surgery can be cyclic only if $r\in \mathds{Z}$. Moreover, there are at most two such integers, and if
there are two then they must be successive.

Although Example \ref{pretzel knot} shows that Theorem \ref{Cyclic} is not true for $SU(2)$-cyclic or $SO(3)$-cyclic surgeries, we have the following analogous result, which is the main theorem of this paper.

\begin{thm}\label{main theorem}
 Consider a nontrivial knot $K\subset S^{3}$ and two surgeries with coefficients $r_{1}=p_{1}/q_{2}$ and $r_{2}=p_{2}/q_{2}$. We have the following:
\begin {itemize}
\item If $r_{1}, r_{2}$ are both $SU(2)$-cyclic, then $\triangle(r_{1},r_{2})\leq |p_{1}|+|p_{2}|$.
\item If $r_{1}, r_{2}$ are both $SO(3)$-cyclic, then $2\triangle(r_{1},r_{2})\leq |p_{1}|+|p_{2}|$.
\end{itemize}
\end{thm}

Combining this theorem with Lemma \ref{homomorphism lift}, we get the following corollaries.

\begin{cor}\label{cor1}

Suppose $r_{1},r_{2}$ are both $SU(2)$-cyclic. If $p_{1}$ is odd, then $2\triangle(r_{1},r_{2})\leq 2|p_{1}|+|p_{2}|$.
 If $p_{1},p_{2}$ are both odd, then $2\triangle(r_{1},r_{2})\leq |p_{1}|+|p_{2}|$.

\end{cor}

\begin{cor}\label{cor2}
If $r_{1},r_{2}$ on $K$ are both $SO(3)$-cyclic surgeries, then $r_{1}r_{2}>0$. If $r_{1},r_{2}$ on $K$ are both $SU(2)$-cyclic surgeries, then $r_{1}r_{2}>0$ unless $r_{1}$ and $r_{2}$ are both even integers.
\end{cor}

\begin{cor} \label{cor3}
For a nontrivial surgery on a nontrivial amphichiral knot $K$ with coefficient $r$, we have the following:
\begin{itemize}
\item It can never be $SO(3)$-cyclic.
 \item If it is $SU(2)$-cyclic, then $r$ is an even integer and some $\frac{r}{2}$-th root of unity is a root of $\Delta_{K}$ (the Alexander polynomial of $K$).
 \end{itemize}
\begin{rmk}
Actually, we haven't found any examples of $SU(2)$-cyclic surgeries on an amphichiral knot. It would be interesting to know whether there exists such a surgery.
\end{rmk}
\end{cor}

We know that $\Delta_{K}(1)=\pm1$ for any knot $K$ while $\Phi_{p}(1)=p$ for any prime number $p$ ($\Phi_{p}$ is the $p$-th cyclotomic polynomial). Therefor the Alexander polynomial $\Delta_{K}$ never has the $p$-th root of unity as its root. We get:

\begin{ex}
If $p$ is prime, then the $2p$-surgery on an non-trivial amphichiral knot is not $SU(2)$-cyclic.
\end{ex}
\begin{rmk}
In \cite{KM}, Kronheimer and Mrowka asked whether there exists $SU(2)$-cyclic surgery with coefficient $3$ or $4$. We see that there exist no such surgeries for nontrivial amphichiral knots.
\end{rmk}

Using the criterion in Corollary \ref{cor3}, we checked the amphichiral knots with crossing number $\leq 10$ and get:

\begin{ex}
All the nontrivial amphichiral knots with crossing number $\leq10$ except perhaps $8_{18}$ and $10_{99}$ in Rolfsen's knot table admit no $SU(2)$-cyclic surgeries. For $8_{18}$ and $10_{99}$, we have no examples of $SU(2)$-cyclic surgeries.
\end{ex}

\begin{cor}\label{cor4}
Given a nontrivial knot $K$ and an integer $q$, there exist at most finitely many $p\in \mathbb{Z}$ such that $(p,q)=1$ and the $\frac{p}{q}$-surgery on $K$ is $SO(3)$-cyclic. For the $SU(2)$ case, the only possible exception is when $q=1$ and infinitely many even $p$.

In particular, any nontrivial knot admits only finitely many integer $SO(3)$-cyclic surgeries and only finite many odd $SU(2)$-cyclic surgeries.
\end{cor}

The paper is organized as follows: in section 2, we review some preliminaries and basic constructions related to holonomy perturbations. In section 3, we prove the main theorem and the corollaries.

\bigskip\noindent\textbf{Acknowledgement} The author wishes to thank Nathan Dunfield, Cameron Gordon, Yi Ni and Yi Liu for valuable discussions and comments. The author is especially grateful to Ciprian Manolescu for inspiring conversations and helpful suggestions in writing this paper.

\section{Preliminaries}

In this section, we review the basic facts about holonomy perturbations. Most details can be found in \cite{KM} and \cite{FLOER}. The constructions are very similar to \cite{Jianfeng}, but for completeness, we review them again here.

Consider the closed manifold $K(0)$. We have $b_{1}(K(0))=1$. Let $E$ be the rank 2 unitary bundle over $K(0)$ with $c_{1}(E)$ the Poincar\'{e} dual of the meridian $m$ of $K$. Let $\mathfrak{g}_{E}$ be the bundle whose sections are traceless, skew-hermitian endomorphisms of $E$. Let $\mathcal{A}$ be the affine space of $SO(3)$ connections of $\mathfrak{g}_{E}$. Let $\mathcal{G}$ be the group of gauge transformations on $E$ with determinant 1. Notice that $\mathcal{G}$ is slightly smaller than the $SO(3)$-gauge transformation group of $\mathfrak{g}_{E}$.

Fix a reference connection $A_{0}$ on $\mathfrak{g}_{E}$. Then for any connection $A$ on $\mathfrak{g}_{E}$, $A-A_{0}$ can be identified with $\omega\in \Omega^{1}(\mathfrak{g}_{E})$. We have the Chern-Simons functional:
$$CS:\mathcal{A}\rightarrow \mathds{R}$$
$$CS(A)=\frac{1}{4}\int_{X_{0}} Tr(2\omega\wedge F_{A_{0}}+\omega\wedge d\omega+\frac{2}{3}\omega\wedge\omega\wedge\omega)$$
Here $F_{A_{0}}$ is the curvature of $A_{0}$.

The critical points of the Chern-Simons functional are the flat connections.

Floer introduced the holonomy perturbations as follows. Take a function $\phi: SU(2)\rightarrow\mathds{R}$ which is invariant under conjugation. Then it is uniquely determined by the even, $2\pi-$periodic function:
 \begin{equation}\label{diag}
 f(x) :=
\phi\left(\begin{array} {lr}
 e^{ix} & 0 \\
 0&e^{-ix}
\end{array}\right)\end{equation}
 Let $D$ be a compact 2-manifold with boundary. Consider an embedding  $D\times S^{1}$ in $K(0)$ such that $\mathfrak{g}_{E}$ is trivial over it. Fix a trivialization of $\mathfrak{g}_{E}$ over $D\times S^{1}$ and take a 2-form $\mu$ which is supported in the interior of $D$ with integral 1. Using the trivialization, we can lift $A$ to a connection $\bar{A}$ on the trivialized $SU(2)$-bundle $\widetilde{P}$ over $D\times S^{1}$. We consider the functional:
$$\Phi : \mathcal{A}\rightarrow\mathds{R}$$
\begin{equation}\label{holonomy perturbation} \Phi(A):=\int_{p\in D}\phi(\text{Hol}_{\{p\}\times S^{1}}(\bar{A}))\mu(p)\end{equation}
Here $\text{Hol}_{\{p\}\times S^{1}}$ is the holonomy along $\{p\}\times S^{1}$.

We decompose $K(0)$ into three parts: $(S^{3}-N(K))\mathop{\cup}\limits_{\{0\}\times l\times m}([0,1]\times l\times m) \mathop{\cup}\limits_{\{1\}\times l\times m}  (D^{2}\times m)$. We have meridians and longitudes on both side of the thicken torus. Denote them by $m_{0}$, $l_{0}$, $m_{1}$, $l_{1}$ respectively.

We should be careful that $m_{0}$ is the meridian of the knot complement but $m_{1}$ the longitude of the attached solid torus. Also $l_{0}$ is the longitude of the knot complement but the $l_{1}$ is the meridian of the attached solid torus.

For our purpose, we will do two types of perturbations:
\begin{itemize}
\item Set $D\cong D^{2}$ and $i_{1}(D^{2}\times S^{1})= (D^{2}\times m)\subset K(0)$. That means we use the holonomy along $m$ to do the perturbation. We denote this perturbation by $\Phi_{1}$
\item Set $D\cong m\times[0,1]$ ($D$ is an annulus ) and $i_{2}(D\times S^{1})=(m\times [0,1])\times l\subset K(0)$. That means we embed a thickened torus and use the holonomy along $l$ to do the perturbation. We denote this perturbation by $\Phi_{2}$.
\end{itemize}

We choose a trivialization of $\mathfrak{g}_{E}$ over $(D^{2}\times m)\cup (m\times [0,1]\times l)$ and use it to  lift the connection $A$ to a $SU(2)$-connection $\overline{A}$ on $\widetilde{P}$. Now use Formula (\ref{holonomy perturbation}) and consider the perturbed Chern-Simons functional $\widehat{CS}=CS+\Phi_{1}+\Phi_{2}: \mathcal{A}\rightarrow \mathds{R}$.

The following theorem was first proved in \cite{KM3}:
\begin{thm}[Kronheimer, Mrowka \cite{KM3}] \label{existence of critical points} If $K$ is a nontrivial knot, then for any holonomy perturbation, the perturbed Chern-Simons functional $\widehat{CS}$ over $K(0)$ always has at least one critical point.
\end{thm}

\begin{rmk}
The proof of this theorem is highly nontrivial. It combines Gabai's result about taut foliation in \cite{Gabai}, Eliashberg-Thurston's theorem about symplectic filling in \cite{Eliashberg} and \cite{Eliashberg-Thurston}, Taubes's result about the Seiberg-Witten invariants of the symplectic four-manifold in \cite{Taubes}, Feehan and Leness's work about Witten's conjecture in \cite{FL} and Kronheimer-Mrowka's work about the refinement of Eliashberg-Thurston's theorem in \cite{KM3}.
\end{rmk}

The critical points can be completed determined:

\begin{lem}\label{critical point}
If $A\in\mathcal{A}$ is a critical point of $\widehat{CS}$, then:

\begin{itemize}
\item $A$ is flat on $S^{3}-N(K)\subset K(0)$.
\item We can choose suitable trivialization of the $SU(2)$-bundle $\widetilde{P}$ such that:
\\ $Hol_{m_{0}}(\overline{A})=\left(\begin{smallmatrix}
 e^{i\theta_{0}} & 0 \\
 0&e^{-i\theta_{0}}
\end{smallmatrix}\right),Hol_{m_{1}}(\overline{A})=\left(\begin{smallmatrix}
 e^{i\theta_{1}} & 0 \\
 0&e^{-i\theta_{1}}
\end{smallmatrix}\right), Hol_{l_{0}}(\overline{A})=\left(\begin{smallmatrix}
 e^{i\eta_{0}} & 0 \\
 0&e^{-i\eta_{0}}
\end{smallmatrix}\right)$
\\and $Hol_{l_{1}}(\overline{A})=\left(\begin{smallmatrix}
 e^{i\eta_{1}} & 0 \\
 0&e^{-i\eta_{1}}
\end{smallmatrix}\right)$.
\item $\eta_{0}=\eta_ {1}=-f_{2}'(\theta_{1})+2\mathds{Z}\pi$ and  $\theta_{0}-\theta_{1}=-f_{1}'(\eta_{0})+2\mathds{Z}\pi$.

\end{itemize}
\end{lem}
\begin{rmk}
Recall that we chose $\phi_{i}:SU(2)\rightarrow \mathds{R}$ to define the perturbation $\Phi_{i}$ ($i=1,2$), which gives us $f_{i}:\mathds{R}\rightarrow \mathds{R}$ by formula (\ref{diag}).
\end{rmk}
\begin{proof} [Proof of Lemma \ref{critical point}] By Lemma 4 in \cite{FLOER} and Lemma 2.2 in \cite{Jianfeng}, $A$ is flat on $S^{3}-N(K)\subset K(0)$ and near $(m\times l\times \{0\})\cup (m\times l\times \{1\})$. Moreover, we can choose a suitable trivialization of $\widetilde{P}$ such that $Hol_{m_{0}}(\overline{A})=\left(\begin{smallmatrix}
 e^{i\theta_{0}} & 0 \\
 0&e^{-i\theta_{0}}
\end{smallmatrix}\right),Hol_{m_{1}}(\overline{A})=\left(\begin{smallmatrix}
 e^{i\theta_{1}} & 0 \\
 0&e^{-i\theta_{1}}
\end{smallmatrix}\right), Hol_{l_{0}}(\overline{A})=\left(\begin{smallmatrix}
 e^{i\eta_{0}} & 0 \\
 0&e^{-i\eta_{0}}
\end{smallmatrix}\right), Hol_{l_{1}}(\overline{A})=\left(\begin{smallmatrix}
 e^{i\eta_{1}} & 0 \\
 0&e^{-i\eta_{1}}
\end{smallmatrix}\right)$ and $\theta_{0}-\theta_{1}=-f_{1}'(\eta_{0})+2\mathds{Z}\pi$. Also, we can choose another trivialization of  $\widetilde{P}$ such that $Hol_{m_{1}}(\overline{A})=\left(\begin{smallmatrix}
 e^{i\theta_{1}'} & 0 \\
 0&e^{-i\theta_{1}'}
\end{smallmatrix}\right), Hol_{l_{1}}(\overline{A})=\left(\begin{smallmatrix}
 e^{i\eta_{1}'} & 0 \\
 0&e^{-i\eta_{1}'}
\end{smallmatrix}\right)$ and $\eta_ {1}'=-f_{2}'(\theta_{1}')+2\mathds{Z}\pi$. Since different trivialzations give the same holonomy modulo conjugation. We have $(\theta_{1}',\eta_{1}')=\pm(\theta_{1},\eta_{1})$. Since $f_{2}'$ is an odd function, we have $\eta_ {1}=-f_{2}'(\theta_{1})+2\mathds{Z}\pi$.\end{proof}

Now suppose $A$ is a critical point.
Since $\mathfrak{g}_{E}$ is trivial over $\pi_{1}(S^{3}-N(K))$, we fix a trivialization of $\mathfrak{g}_{E}|_{S^{3}-N(K)}$. Using this trivialization, we lift the connection $A$ to a $SU(2)$-connection $\widetilde{A}$ over $S^{3}-N(K)$.

  \begin{rmk}This trivialization of $\mathfrak{g}_{E}|_{\pi_{1}(S^{3}-N(K))}$ does not agree with the trivialization over $(D^{2}\times m)\cup (m\times [0,1]\times l)$ (the trivialization which we chose to define the holonomy perturbations) on the torus boundary. They differ by a map $f:\partial(S^{3}-N(K))\rightarrow SO(3)$ such that \begin{equation}\label{gluing}
f_{*}(m)=1\in \pi_{1}(SO(3)), f_{*}(l)=-1\in \pi_{1}(SO(3))\end{equation}
\end{rmk}

By taking the holonomy of $\widetilde{A}$, we get a representation $\rho:\pi_{1}(S^{3}-N(K))\rightarrow SU(2)$.

\begin{defi}
We define a subset $R_{K}$ of $(\mathds{R}/2\pi \mathds{Z})\bigoplus(\mathds{R}/2\pi \mathds{Z})$ as follows:
$$
\{(\theta,\eta)|\exists \rho:\pi_{1}(S^{3}-N(K))\rightarrow SU(2) \text{ s.t. }\rho(m)=\left(\begin{smallmatrix}
 e^{i\theta} & 0 \\
 0&e^{-i\theta}
\end{smallmatrix}\right); \rho(l)=\left(\begin{smallmatrix}
 e^{i\eta} & 0 \\
 0&e^{-i\eta}
\end{smallmatrix}\right) \}
$$

\end{defi}

We can also describe $R_{K}$ as:
$$
\{(\theta,\eta)|\exists \text{ flat } \text{connection } \widetilde{A} \text{ over } S^{3}-N(K)\text{ s.t.
Hol}_{m}(\widetilde{A})=\left(\begin{smallmatrix}
 e^{i\theta} & 0 \\
 0&e^{-i\theta}
\end{smallmatrix}\right); \text{Hol}_{l}(\widetilde{A})=\left(\begin{smallmatrix}
 e^{i\eta} & 0 \\
 0&e^{-i\eta}
\end{smallmatrix}\right) \}
$$

\begin{nota}
Let $S\subset (\mathds{R}/2\pi \mathds{Z})\bigoplus(\mathds{R}/2\pi \mathds{Z})$ be a subset. If $h$ is a function with period $2\pi$, we denote the set $\{(\theta,\eta+h(\theta))| (\theta,\eta)\in S \}$ by $S+(*,h)$ and the set $\{(\theta+h(\eta),\eta)| (\theta,\eta)\in S \}$ by $S+(h,*)$.
We also denote the set $\{(\theta+a,\eta+b)| (\theta,\eta)\in S \}$ by $S+(a,b)$ for constant $a,b$.

\end{nota}

The following lemma is proved in \cite{KM}. We change the statement a little. For completeness, we give the proof here.

\begin{lem}\label{property of pillowcase}
$R_{K}$ has the following properties:
\begin{itemize}
\item 1)Any point in $R_{K}$ off the line $\{\eta=2\pi\mathds{Z}\}$ gives some non-cyclic representation.
\item 2)$R_{K}$ is a closed subset of $(\mathds{R}/2\pi \mathds{Z})\bigoplus(\mathds{R}/2\pi \mathds{Z})$.
\item 3)$R_{K}=R_{K}+(\pi,0)$.
\item 4)$R_{K}\cap\{\theta=k\pi\}=(k\pi,2k'\pi), (k,k'\in \mathds{Z}). $
\item 5)$\exists \epsilon>0$ such that $\forall k\in \mathds{Z},R_{K}\cap \{\theta\in [k\pi-\epsilon,k\pi+\epsilon]\}\cap \{\eta\neq 2\mathds{Z}\pi\}=\emptyset.$\end{itemize}
\end{lem}

\begin{proof}

1) Any point in $R_{K}$ gives a representation $\rho:\pi_{1}(S^{3}-N(K))\rightarrow SU(2)$. If $\rho$ is cyclic, then $\rho$ factors through $H_{1}(S^{3}-N(K))$, which implies that $\rho(l)=1\in SU(2)$.

2) $R_{K}$ is closed because $\pi_{1}(S^{3}-K)$ is finitely generated and $SU(2)$ is compact.

3) We have a map $\rho_{0}:\pi_{1}(S^{3}-K)\rightarrow H_{1}(S^{3}-K)\rightarrow \mathds{Z}_{2}\subset SU(2)$ such that $\rho_{0}(m)=-1\in SU(2)$ and $ \rho_{1}(l)=1\in SU(2)$. For any homomorphism $\rho:\pi_{1}(S^{3}-N(K))\rightarrow SU(2)$, we can multiply it by $\rho_{0}$ to get another representation $\rho'$ such that $\rho'(l)=\rho(l)$ and $\rho'(m)=-\rho(m)$. By definition of $R_{K}$, we get $R_{K}=R_{K}+(\pi,0)$.

4) Suppose $\rho$ is given by a point with $\theta=0$, then $\rho(m)=1\in SU(2)$ and $\rho$ factors through $\pi_{1}(S^{3})$, which is a trivial. We get $\rho(l)=1$ and $\eta=2k'\pi$. For the case $\theta=\pi$, we use 3).

5) Look at a small neighborhood $U$ of $(0,0)\in R_{K}$ in $(\mathds{R}/2\pi \mathds{Z})\bigoplus(\mathds{R}/2\pi \mathds{Z})$. The point $(0,0)$ is given by the restriction of the trivial representation $\rho_{1}$. The deformations of $\rho_{1}$ are governed by $H^{1}(\pi_{1}(S^{3}-K),\mathds{R}^{3})\cong \mathds{R}^{3}$. But every vector in this $\mathds{R}^{3}$ can be realized by the some reducible representation. We see that in a small neighborhood of $\rho_{1}$, all the representations are reducible. Thus $U\cap R_{K}\cap \{\eta\neq 2\mathds{Z}\pi\}=\emptyset$ if $U$ is small enough. Use 4) and the compactness of $R_{K}$, we prove 5) for the case $k=0$. Then we use 3) to prove the case $k=1$.

\end{proof}

\begin{lem}
If $A$ is a critical point of the perturbed Chern-Simons functional, then $(\theta_{0},\eta_{0})\in R_{K}+(0,-\pi)$.
($\theta_{0}$ and $\eta_{0}$ are defined in Lemma $\ref{critical point}$)
\end{lem}

\begin{proof}
On the torus $m\times l\times \{0\}$, we use two different trivializations to lift $A$ to two $SU(2)$ connections $\overline{A}$ and $\widetilde{A}$. Because of formula (\ref{gluing}), we see that   $(\text{Hol}_{m_{0}}(\widetilde{A}),\text{Hol}_{l_{0}}(\widetilde{A}))\in SU(2)\times SU(2)$ is conjugate with  $(\text{Hol}_{m_{0}}(\overline{A}),-\text{Hol}_{l_{0}}(\overline{A}))$. So we can change the trivialization of $\mathfrak{g}_{E}$ over $S^{3}-N(K)$ so that $(\text{Hol}_{m_{0}}(\overline{A}),-\text{Hol}_{l_{0}}(\overline{A}))=(\text{Hol}_{m_{0}}(\widetilde{A}),\text{Hol}_{l_{0}}(\widetilde{A}))$.   Then we use the second description of $R_{K}$.\end{proof}

Combining this lemma and Lemma \ref{critical point}, we get:
\begin{lem}\label{boundary value}
If $A$ is a critical point of the perturbed Chern-Simons functional, then $(\theta_{1},\eta_{1})\in (R_{K}+(0,-\pi)+(f_{1}',*))\cap\{\eta=-f_{2}'(\theta)\}$.
\end{lem}

\section{Proof of the Theorem and Corollaries }
\subsection{Proof of the main theorem}
Now suppose $K\subset S^{3}$ is a nontrivial knot. Denote the set $R_{K}\cap\{\eta\neq2\mathds{Z}\pi\}$ by $R^{*}_{K}$. For $r=\frac{p}{q}$, we define the subsets $S(r)$ and $\widehat {S}(r)$ to be:
$$
S(r):= \{(\theta,\eta)|(p\theta+q\eta)\in2\mathds{Z}\pi \text{ or } (p\theta+p\pi+q\eta)\in2\mathds{Z}\pi\}
$$
$$\widehat{S}(r):=\{(\theta,\eta)|(p\theta+q\eta)\in\mathds{Z}\pi\}$$

\begin{rmk}
When $p$ is odd, we have $\widehat{S}(r)=S(r)$.
\end{rmk}
\begin{lem}\label{not intersect}
If $r$ is an $SU(2)$-cyclic surgery, then $R^{*}_{K}\cap S(r)=\emptyset$.
If $r$ is an $SO(3)$-cyclic surgery, then $R^{*}_{K}\cap \widehat{S}(r)=\emptyset$.
\end{lem}

\begin{proof}
If $(\theta,\eta)\in R^{*}_{K}$ satisfies $p\theta+q\eta\in2\mathds{Z}\pi$, then it gives a representation $\rho:\pi_{1}(S^{3}-N(K))\rightarrow SU(2)$ such that $\rho(pm+ql)=1\in SU(2)$. Thus $\rho$ factors through $\pi_{1}(K(r))$. By (1) of Lemma \ref{property of pillowcase}, $\rho$ is non-cyclic. We get the contradiction since $r$ is a $SU(2)$-cyclic surgery. We see that $R^{*}_{K}\cap \{(\theta,\eta)|(p\theta+q\eta)\in2\mathds{Z}\pi\}=\emptyset$. By (3) of Lemma\ref{property of pillowcase}, we have $R^{*}_{K}+(\pi,0)=R^{*}_{K}$. Thus we also have $R^{*}_{K}\cap \{(\theta,\eta)|(p\theta+p\pi+q\eta)\in2\mathds{Z}\pi\}=\emptyset$. We proved the first assertion. The second assertion can be proved similarly.
\end{proof}

Since we are considering the subsets of $(\mathds{R}/2\pi \mathds{Z})\bigoplus(\mathds{R}/2\pi \mathds{Z})$, it will be convenient to fix a region $W=\{(\theta,\eta)|\theta\in (-\infty,\infty), \eta\in [0,2\pi]\}\subset \mathds{R}^{2}$. We define $W^{*}$ to be $\{(\theta,\eta)|\theta\in (-\infty,\infty), \eta\in (0,2\pi)\}$. We can work in $W$ and $W^{*}$ and then project to $(\mathds{R}/2\pi \mathds{Z})\bigoplus(\mathds{R}/2\pi \mathds{Z})$.

For two different numbers $r_{1}=\frac{p_{1}}{q_{1}},r_{2}=\frac{p_{2}}{q_{2}}$. We define another two numbers:
$$d_{1}(r_{1},r_{2}) = \left\{
  \begin{array}{l l}
    \frac{2\pi|p_{1}|}{\Delta(r_{1},r_{2})} & \quad \text{if $p_{2}$ is even}\\
     \frac{\pi|p_{1}|}{\Delta(r_{1},r_{2})} & \quad \text{if $p_{2}$ is odd}
  \end{array} \right.;
  d_{2}(r_{1},r_{2}) = \left\{
  \begin{array}{l l}
    \frac{2\pi|p_{2}|}{\Delta(r_{1},r_{2})} & \quad \text{if $p_{1}$ is even}\\
     \frac{\pi|p_{2}|}{\Delta(r_{1},r_{2})} & \quad \text{if $p_{1}$ is odd}
  \end{array} \right.$$

 The intersection $S(r_{i})\cap W^{*} $ are just some line segments of slope $-r_{i}$ and $S(r_{1})\cap S(r_{2})\cap W^{*} $ consists of isolated points. We say two intersection points in $S(r_{1})\cap S(r_{2})\cap W^{*}$ are adjacent in $S(r_{i})$ $(i=1,2)$ if they lie in the same component of  $S(r_{i})\cap W^{*} $ and there is no intersection point between them. We define two intersection points in $\widehat{S}(r_{1})\cap \widehat{S}(r_{2})\cap W^{*}$ to be adjacent in $\widehat{S}(r_{i})$ in a similar way.

The following lemma is easy to prove:

\begin{lem}\label{intersection}
(1) If two intersection points $(\theta,\eta),(\theta',\eta')\in S(r_{1})\cap S(r_{2})\cap W^{*} $ are adjacent in $S(r_{i})$, then $|\eta-\eta'|=d_{i}(r_{1},r_{2}) $ $(i=1,2)$.

(2) If two intersection points $(\theta,\eta),(\theta',\eta')\in \widehat{S}(r_{1})\cap \widehat{S}(r_{2})\cap W^{*} $ are adjacent in $\widehat{S}(r_{i})$, then $|\eta-\eta'|=\frac{\pi|p_{i}|}{\Delta(r_{1},r_{2})} $ $(i=1,2)$.

(3) For $(\theta,\eta)\in S(r_{1})\cap S(r_{2})\cap W^{*}$, if $\eta> d_{i}(r_{1},r_{2})$, then we can find $(\theta',\eta')\in S(r_{1})\cap S(r_{2})\cap W^{*}$ such that they are adjacent in $S(r_{i})$ and $\eta'<\eta$. If $\eta< 2\pi-d_{i}(r_{1},r_{2})$, then we can find $(\theta',\eta')\in S(r_{1})\cap S(r_{2})\cap W^{*}$ such that they are adjacent in $S(r_{i})$ and $\eta'>\eta$.

(4) For $(\theta,\eta)\in \widehat{S}(r_{1})\cap \widehat{S}(r_{2})\cap W^{*}$, if $\eta>\frac{\pi|p_{i}|}{\Delta(r_{1},r_{2})}$, then we can find $(\theta',\eta')\in \widehat{S}(r_{1})\cap \widehat{S}(r_{2})\cap W^{*}$ such that they are adjacent in $\widehat{S}(r_{i})$ and $\eta'<\eta$. If $\eta< 2\pi-\frac{\pi|p_{i}|}{\Delta(r_{1},r_{2})}$, then we can find $(\theta',\eta')\in \widehat{S}(r_{1})\cap \widehat{S}(r_{2})\cap W^{*}$ such that they are adjacent in $\widehat{S}(r_{i})$ and $\eta'>\eta$.
\end{lem}

Now we can start the proof of our main theorem:

\begin{proof} [Proof of Theorem \ref{main theorem}]
 Let $r_{1},r_{2}$ be two $SU(2)$-cyclic surgeries. Since the theorem is trivial when $r_{1}=r_{2}$, we always assume that $r_{1}\neq r_{2}$. By Theorem \ref{|r|>2}, we have $|r_{i}|>2$. Moreover, when $r_{1}$ or $r_{2}$ equals $\frac{1}{0}$, the identities in the theorem and corollaries can be easily deduced from Theorem $\ref{|r|>2}$. Thus we can assume $p_{i}\neq0$ and $q_{i}\neq 0$. Suppose $d_{1}(r_{1},r_{2})+d_{2}(r_{1},r_{2})<2\pi$. By Lemma \ref{not intersect} and (4) of Lemma \ref{property of pillowcase}, we have $R_{K}^{*}\cap (S(r_{1})\bigcup S(r_{2})\bigcup \{\theta=k\pi\})=\emptyset$.  We will construct a broken line $L:[-1,1]\rightarrow W$ such that $Im(L)\subset S(r_{1})\bigcup S(r_{2})\bigcup \{\theta=k\pi\}$. There are two cases:

(1) Suppose $r_{1}<-2<2< r_{2}$. Let $L(0)=(0,\pi)$. Then as $t$ increases, $L$ first goes up along $\theta=0$ to $(0,2\pi)$. Since $(0,2\pi)\in S(r_{2})$, $L$ can go down along $S(r_{2})$ to the lowest intersection point $(\theta_{1},\eta_{1})\in S(r_{1})\cap S(r_{2})\cap W^{*}$ on this line segment of $S(r_{2})$. By (3) of Lemma \ref{intersection}, we have $\eta_{1}\leq d_{2}(r_{1},r_{2})$. By our assumption, we have $\eta_{1}<2\pi-d_{1}(r_{2},r_{2})$.

Again by Lemma \ref{intersection}, $(\theta_{1},\eta_{1})$ is not the highest intersection point in the component of  $S(r_{1})\cap W^{*}$ which contains it. Thus $L$ can go along $S(r_{1})$ to the highest intersection point. Notice that this point is still in $W^{*}$. After that, $L$ again goes along $S(r_{2})$ to the lowest intersection point. Repeat this procedure until $L$ hits the line ${\theta=\pi}$. Then $L$ goes along ${\theta=\pi}$ to the point $(\pi,\pi)$. We have defined $L(t)$ for $t\in [0,1]$. Reflecting along $(0,\pi)$, we can define $L(t)$ for $t\in [-1,0]$.

\begin{figure}[htbp]
\centering
\includegraphics[angle=0,width=1.0\textwidth]{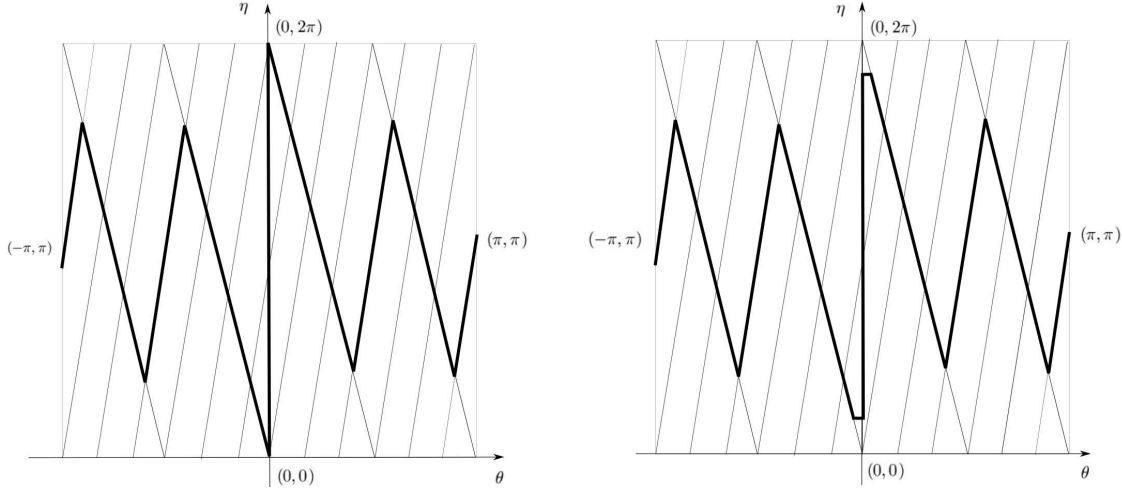}
\caption{$L$ (left) and $\widehat{L}$ (right) when $r_{1}=-3,r_{2}=4$} \label{meinv}
\end{figure}

(2) Suppose $r_{1},r_{2}$ are of the same sign. We do the case $2<r_{1}<r_{2}$ and the other case is similar. Set $L(0)=(0,\pi)$ and let $L$ goes along $\theta=0$ to $(0,2\pi)$. Then $L$ moves down alone $S(r_{1})$ to the lowest intersection point in $W^{*}$. After that $L$ moves along $S(r_{2})$ to the highest intersection point. The difference from case (1) is that we repeat this procedure until $L$ intersects the line segment $l\subset S(r_{1})$ which passes through $(\pi,0)$. It is easy to see that this happens before $L$ hits $\theta=\pi$. Then $L$ goes along $l$ to $(\pi,0)$ and then goes along $\theta=\pi$ to $(\pi,\pi)$.  By reflecting along the point $(0,\pi)$, we define $L(t)$ for all $t\in[-1,1]$.
\begin{figure}[htbp]
\centering
\includegraphics[angle=0,width=1.0\textwidth]{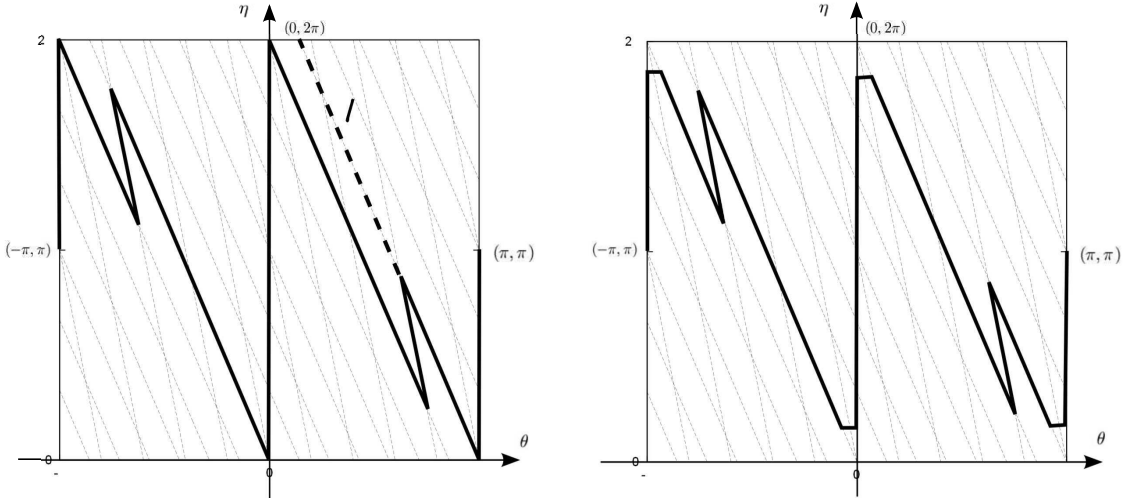}
\caption{L (left) and $\widehat{L}$ (right) when $r_{1}=\frac{7}{3},r_{2}=5$} \label{meinv}
\end{figure}

We denote the image of $L$ by $\mathit{Im}(L)\subset W$. In both cases, we have $\mathit{Im}(L)\subset S(r_{1})\cup S(r_{2})\cup \{\theta=k\pi\}$. Thus $R_{K}^{*}\cap \mathit{Im}(L)=\emptyset$. The image of $L$ intersects the line $\eta=0$ and $\eta=2\pi$ at $(0,0), (0,2\pi)$ in case (1) and at $(0,0), (0,2\pi), (\pi,0),(-\pi,2\pi)$ in case (2). We need to do small modification around these points. Take the point $(0,2\pi)$ for example. We choose a small neighborhood $U$ of $(0,2\pi)$ and remove $\mathit{Im}(L)\cap U$. Then we replace it with a short horizontal line segment $\eta=2\pi-\varepsilon$. By 5) of Lemma \ref{property of pillowcase}, after doing this modification, we still get a map $\widehat{L}:[-1,1]\rightarrow W^{*} $ such that $\mathit{Im}(\widehat{L})\cap R_{K}=\emptyset$. Moreover, $\mathit{Im}(\widehat{L})$ is symmetric under the reflection about $(0,\pi)$. Suppose $\widehat{L}(t)=(\theta(t),\eta(t))$.  By the compactness of $R_{K}$, there exists a small neighborhood $N$ of $\mathit{Im}(\widehat{L})$ such that $N\cap R_{K}=\emptyset$.

In case (1), ``$\widehat {L}$ goes forward'', which means that $\theta(t)\geq\theta(t')$ if $t\geq t'$. Since $(0,\pi),(\pm\pi,\pi)$ $\in \mathit{Im}(\widehat{L})$ and $\mathit{Im}(\widehat{L})$ is symmetric under the reflection of $(0,\pi)$, there exists a smooth odd function $g_{2}$ with period $2\pi$ such that the graph $\{\eta=g_{2}(\theta)\}$ is contained in $N+(0,-\pi)$. Thus $R_{K}+(0,-\pi)$ does not intersect the graph of $g_{2}$. We can choose an even, $2\pi$-periodic function $f_{2}$ such that $f'_{2}=-g_{2}$. Set $f_{1}\equiv0$ and use Lemma \ref{boundary value}. We see that $\widehat{CS}$ has no critical point, which contradicts Theorem \ref{existence of critical points}.
\begin{figure}[htbp]
\centering
\includegraphics[angle=0,width=1.0\textwidth]{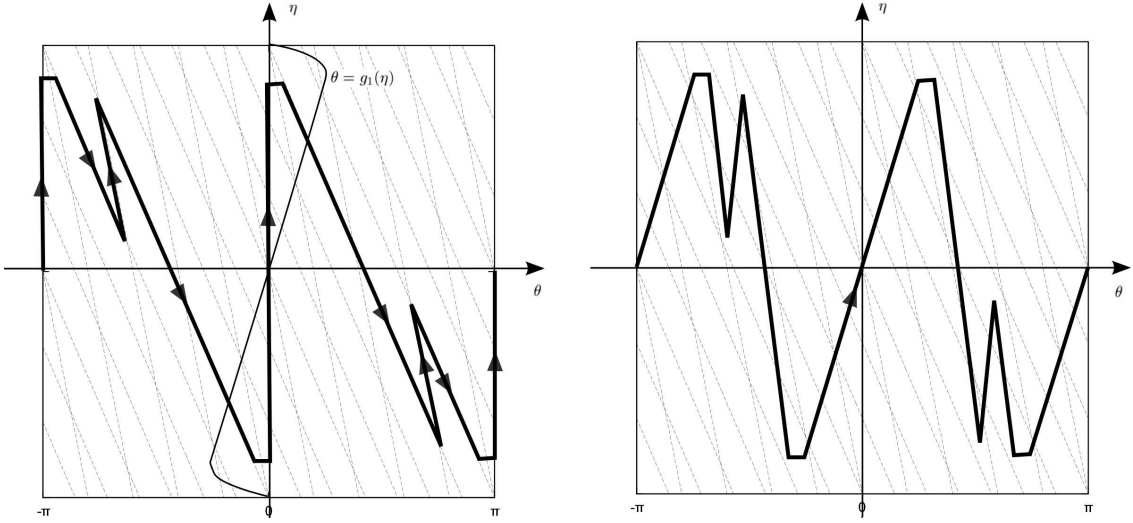}
\caption{The iamge of $\widehat{L}+(0,-\pi)$  (left) and $\widehat{L}+(0,-\pi)+(g_{1},*)$ (right) when $r_{1}=\frac{7}{3},r_{2}=5$} \label{meinv}
\end{figure}

In case (2), $\widehat {L}$ does not always go forward and our argument needs to be modified. Take the case $2<r_{1}<r_{2}$ for example (see Figure 3). By the construction of $\widehat {L}$, there exists a small $\epsilon>0$ such that $\mathit{Im}(\widehat {L})$ is contained in the region $\epsilon<\eta<2\pi-\epsilon$. Choose a number $r_{0}\in (r_{1},r_{2})$. There exist an odd, periodic-$2\pi$ function $\theta=g_{1}(\eta)$ such that $g_{1}(\eta)=\frac{\eta}{r_{0}},\forall \eta\in [\epsilon,2\pi-\epsilon]$. The image of $\widehat {L}$ only consists of the following 4 types of segments:
 \begin{itemize}
 \item i) horizontal line that goes forward,
  \item ii) going down line of slope $-r_{1}$,
  \item iii) going up line of slope     $-r_{2}$,
  \item iv) going up line of slope $+\infty$.
  \end{itemize}
  Therefore, it is not difficult to see that $\mathit{Im}(\widehat{L})+(0,-\pi)+(g_{1},*)$ is a broken line that goes forward (it's just a shearing of $\mathit{Im}(\widehat{L})+(0,-\pi)$). Thus we can find an odd, $2\pi$-periodic function $g_{2}$ such that the graph $\{\eta=g_{2}(\theta)\}$ is contained in $N+(0,-\pi)+(g_{1},*)$. We see that $R_{K}+(0,-\pi)+(g_{1},*)$ does not intersect the graph $\{\eta=g_{2}(\theta)\}$. We can find even, $2\pi-$periodic  function $f_{1},f_{2}$ such that $f_{1}'=g_{1}, f_{2}'=-g_{2}$. Using Lemma \ref{boundary value} and Theorem \ref{existence of critical points}, we get the contradiction again.

The $SO(3)$-cyclic case can be proved similarly by considering $\widehat {S}(r_{i})$ instead of $S(r_{i})$.\end{proof}

\begin{rmk}
Actually, we have proved that if $r_{1},r_{2}$ are both $SU(2)$-cyclic, then $d_{1}(r_{1},r_{2})+d_{2}(r_{1},r_{2})\geq 2\pi$. When $p_{i}$ is odd, this gives the conclusions of Corollary \ref{cor1}.
\end{rmk}

 Corollary \ref{cor1}, Corollary \ref{cor2} and Corollary \ref{cor4} are easy to prove using the main theorem.

\subsection{Relation with the Alexander polynomial}

In this subsection, we will give some relations between the $SU(2)$-cyclic surgeries and the Alexander polynomial and prove Corollary \ref{cor3}.

 Suppose $d_{1}(r_{1},r_{2})+d_{2}(r_{1},r_{2})=2\pi$ (for example $r_{1}=-r_{2}=2k$) and $r_{1},r_{2}$ are both $SU(2)$-cyclic. Let's try to repeat the argument as before. We do the case $r_{2}<0<r_{1}$ and the other cases are similar. Consider $S(r_{i})\subset (\mathds{R}/2\pi \mathds{Z})\bigoplus(\mathds{R}/2\pi \mathds{Z})$ $(i=1,2)$, then $R_{K}^{*}\cap S(r_{i})=\emptyset$. We now construct $L:[-1,1]\rightarrow W$. Set $L(0)=(0,\pi)$ and $L$ goes upwards along $\theta=0$ to $(0,2\pi)$. Then $L$ goes down along $S(r_{2})$ the the lowest intersection point $(\theta_{1},\eta_{1})\in S(r_{1})\cap S(r_{2})\cap W$. After that, $L$ goes up along $S(r_{1})$ to the highest intersection point $(\theta_{2},\eta_{2})\in S(r_{1})\cap S(r_{2})\cap W$. As shown in Figure 4, it is possible that the lowest intersection point in $(\theta_{2},\eta_{2})\in S(r_{1})\cap S(r_{2})\cap W^{*}$ is also the highest one. Thus we can only work in $W$ instead of $W^{*}$. This is different from the case when $d_{1}(r_{1},r_{2})+d_{2}(r_{1},r_{2})<2\pi$. We repeat this procedure and get $L:[-1,1]\rightarrow W$. Then we need to modify $L$ to $\widehat {L}$ whose image is contained in $W^{*}$. The trouble appears: $L$ may contain some points like $(\theta_{0},0) \text{ or } (\theta_{0},2\pi)$ with $\theta_{0}\neq 0 \text{ or} \pm \pi$. In general, we don't have the result like 5) of Lemma 2.8 which allows us to modify $L$ near these points without intersecting $R_{K}$.

\begin{figure}[htbp]
\includegraphics[angle=0,width=1.0\textwidth]{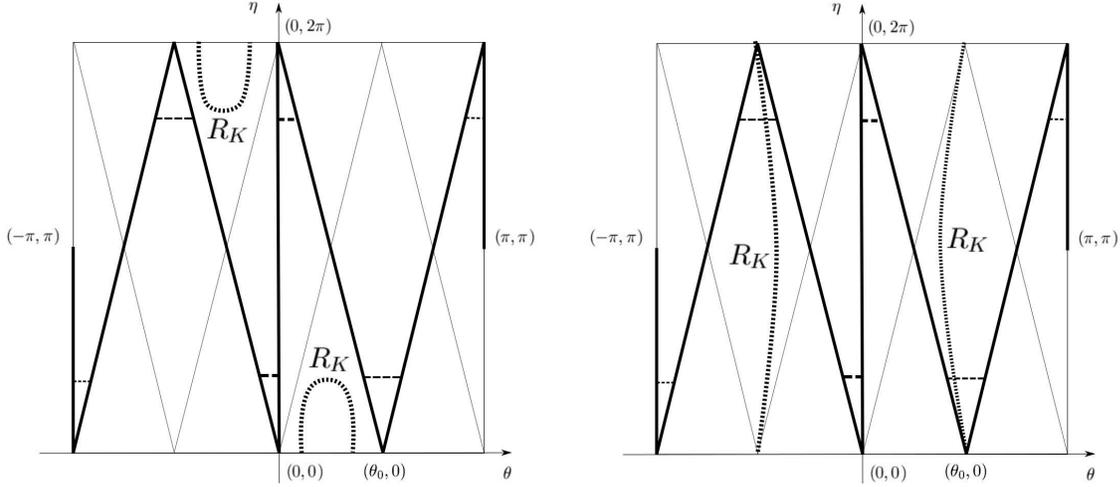}
\caption{When $r_{1}=4,r_{2}=-4$, we can modify $L$ near $(\theta_{0},0)$ in the left picture but we can't modify $L$ in the right picture.} \label{meinv}
\end{figure}

 We just do the $(\theta_{0},0)$ case and the $(\theta_{0},2\pi)$ case is similar. Suppose that we can choose a small neighborhood $U$ of $(\theta_{0},0)$ such that $R_{K}^{*}\cap U=\emptyset$. We just replace $Im({L})\cap U$ by some short, horizontal line $l\subset U\cap W^{*}$. If we can do this for every point in $Im(L)\cap (W\backslash W^{*})$, we can construct $\widehat{L}$ and get the contradiction as before. If we can't do this for some point $(\theta_{0},0)\in S(r_{1})$, then there exist a sequence $(\theta_{n},\eta_{n})\in R_{K}^{*}$ converging to $(\theta_{0},0)$ as $n\rightarrow \infty$. Each $(\theta_{n},\eta_{n})$ gives an irreducible representation $\rho_{n}:\pi_{1}(S^{3}-N(K))\rightarrow SU(2)$. It is easy to see that these representations are also irreducible as $SL(2,\mathds{C})$ representations. By the compactness of $SU(2)$ representation variety, $\rho_{n}$ converge to some $\rho_{0}$  after taking a subsequence. We will get $(\theta_{0},0)\in S(r_{1})$ if we restrict $\rho_{0}$ to the boundary. Recall that we have a representation $\pi_{1}(S^{3}-K)\rightarrow \pm 1\rightarrow SU(2)$ such that $m$ is mapped to $-1$. After multiplying $\rho_{0}$ by this representation if necessary, we get a representation of $\rho'_{0}$ $\pi_{1}(S^{3}-N(K))$ such that $\rho'_{0}(p_{1}m+q_{1}l)=1$. Since $r_{1}$ is an $SU(2)$-cyclic surgery, this representation must be cyclic. In particular, this implies that $\rho_{0}$ is cyclic. Thus we get a sequence of irreducible $SL(2,\mathds{C})$ representations of $\pi_{1}(S^{3}-N(K))$ converging to a reducible $SL(2,\mathds{C})$ representation $\rho_{0}$ with $\rho_{0}(m)=\left(\begin{smallmatrix}
 e^{i\theta_{0}} & 0 \\
 0&e^{-i\theta_{0}}
\end{smallmatrix}\right)$.

We apply the following proposition in \cite{CCGLS}:
\begin{pro}[\cite{CCGLS}]
Let $M$ be the complement of a knot $K$ in a homology 3-sphere. Suppose that $\rho$ is a reducible representation of $\pi_{1}(M)$ such that the character of $\rho$ lies on a component of $\chi(M)$ which contains the character of an irreducible representation. Then $\rho(m)$ has an eigenvalue whose square is a root of $\Delta_{K}$ (the Alexander polynomial of $K$).
\end{pro}

Using this theorem, we see that $e^{2i\theta_{0}}$ is a root of $\Delta_{K}$. Since $(\theta_{0},0)\in S(r_{1})$, we see that $\Delta_{K}$ has a root which is   a $p_{1}$-th root of unity for odd $p_{1}$ and $\frac{p_{1}}{2}$-th root of unity for even $p_{1}$.

By considering the intersection point $(\theta_{0},2\pi)\in Im(L)\cap \{\eta=2\pi\}$, we can get the same conclusion for $p_{2}$. In particular, we get the following:

\begin{pro}
 Suppose that $r_{1}=\frac{p_{1}}{q_{1}},r_{2}=\frac{p_{2}}{q_{2}}$ are two $SU(2)$-cyclic surgeries with  $d_{1}(r_{1},r_{2})+d_{2}(r_{1},r_{2})=2\pi$ and $p_{1},p_{2}$ even, then the Alexander polynomial of $K$ has a root which is either a $\frac{p_{1}}{2}$-th or a $\frac{p_{2}}{2}$-th root of unity.
\end{pro}

Notice that if $K$ is amphichiral, then the $r$-surgery is $SU(2)$-cyclic implies that the $-r$-surgery is also $SU(2)$-cyclic. By Corollary \ref{cor2}, we get $r$ is an even integer and $d_{1}(r,-r)+d_{2}(r,-r)=2\pi$. Therefore, Corollary \ref{cor3} is a straightforward consequence of the proposition above.

\vskip 0.5 truecm

\end{document}